\documentclass[11pt,noindent]{article}
\usepackage{latexsym,amsmath,amsfonts}
\usepackage{hyperref}
\usepackage{setspace}
\newtheorem{theorem}{Theorem}

\newtheorem{proposition}{Proposition}
\newtheorem{corollary}{Corollary}
\newtheorem{lemma}{Lemma}

\newtheorem{problem}{Problem}

\newtheorem{notation}{Notational convention}

\newtheorem{claim}{Claim}

\newtheorem{remark}{Remark}

\numberwithin{theorem}{section}
\numberwithin{lemma}{section}
\numberwithin{proposition}{section}
\numberwithin{corollary}{section}
\numberwithin{claim}{section}

\newcommand{\thmref}[1]{Theorem~\ref{thm:#1}} 
\newcommand{\lemref}[1]{Lemma~\ref{lem:#1}} 
\newcommand{\propref}[1]{Proposition~\ref{prop:#1}} 
\newcommand{\remref}[1]{Remark~\ref{rem:#1}} 
\newcommand{\secref}[1]{Section~\ref{sec:#1}} 
\newcommand{\eqnref}[1]{(\ref{eq:#1})} 

\def\be{\begin{equation} }
\def\ee{ \end{equation}}

\def\ben{\begin{equation*}}
\def\een{\end{equation*}}
\def\bea{\begin{eqnarray}}
\def\eea{\end{eqnarray}}
\def\ee{\end{eqnarray}}
\def\bean{\begin{eqnarray*}}
\def\eean{\end{eqnarray*}}

\newcommand\ignore[1]{}

\def\R{\mathbb{R}} 
\def\Z{\mathbb{Z}} 
\def\N{\mathbb{N}} 

\newcommand{\Ex}[1]{\mathbb{E}\left[#1\right]} 
\newcommand{\Prp}[2]{\mathbb{P}_{#1}\left(#2\right)} 
\newcommand{\Exp}[2]{\mathbb{E}_{#1}\left[#2\right]} 
\newcommand{\Prpwo}[1]{\mathbb{P}_{#1}} 
\newcommand{\Prwo}{\mathbb{P}} 
\newcommand{\Ind}[1]{\mathbb{I}_{#1}} 

\renewcommand{\Pr}[1]{\mathbb{P}\left(#1\right)} 

\newcommand{\bigoh}[1]{O\left(#1\right)}

\newcommand{\theita}[1]{\Theta\left(#1\right)}

\def\sA{\mathcal{A}}
\def\sF{\mathcal{F}}
\def\sG{\mathcal{G}}\def\sH{\mathcal{H}}

\def\sO{\mathcal{O}}
\def\sP{\mathcal{P}}

\newcommand\QED{\ifhmode\allowbreak\else\nobreak\fi
\quad\nobreak$\Box$\medbreak}
\newcommand{\proofstart}{\par\noindent\sl Proof:\rm\enspace}
\newcommand{\proofend}{\QED\par}
\newenvironment{proof}{\proofstart}{\proofend}

\def\qstat{{\sf q}}

\def\Tmix{{\rm T}_{\rm mix}}
\def\Thit{{\rm T}_{\rm hit}}

\def\bG{{\bf G}}
\def\bbD{\mathbb{D}}

\def\bV{{\bf V}}
\def\Co{{\sf Co}}

\begin{document}

\title{On the coalescence time of reversible random walks}
\author{Roberto Imbuzeiro Oliveira\thanks{IMPA, Rio de Janeiro, RJ, Brazil, 22430-040.  Work supported by a {\em Bolsa de Produtividade em Pesquisa} and by a {\em Pronex} grant from CNPq, Brazil.    }} \maketitle
\begin{abstract}Consider a system of coalescing random walks where each individual performs random walk over a finite graph $\bG$, or (more generally) evolves according to some reversible Markov chain generator $Q$. Let $C$ be the first time at which all walkers have coalesced into a single cluster. $C$ is closely related to the consensus time of the voter model for this $\bG$ or $Q$.

We prove that the expected value of $C$ is at most a constant
multiple of the largest hitting time of an element in the state
space. This solves a problem posed by Aldous and Fill and gives
sharp bounds in many examples, including all vertex-transitive
graphs. We also obtain results on the expected time until only
$k\geq 2$ clusters remain. Our proof tools include a new exponential
inequality for the meeting time of a reversible Markov chain and a
deterministic trajectory, which we believe to be of independent
interest.\end{abstract}


(secondary).

\section{Introduction}

Consider a system of continuous-time random walks on a finite
connected graph $\bG$, with a walker starting from each vertex
of $\bG$. Let the walkers evolve independently, except that
any two that occupy the same vertex of $\bG$ at a given time {\em
coalesce} into one (this is made precise in \secref{CRW}).

As time goes by, larger and larger coalesced clusters emerge, until
at a certain random time $C$ only one cluster remains. The question we address here is: how large
can $C$ be in terms of other parameters of $\bG$? This is a natural question which has implications for the so-called {\em voter model on $\bG$}, discussed in \secref{voter} below.

It is instructive to consider what happens in the simple case of
$\bG=K_n$, the complete graph on $n$ vertices. An explicit
calculation \cite[Chapter 14, Sec. 3.3]{AldousFill_RWBook} shows
that:
\begin{equation}\label{eq:distcomplete}\frac{C}{n}\approx \sum_{i=1}^{+\infty}
\frac{Z_i}{i(i+1)}\mbox{, with $\{Z_i\}_{i\geq 1}$ i.i.d.
exponential random variables with mean $1$}.\end{equation} In
particular, $\Ex{C}\sim n$ as $n\to +\infty$. What is remarkable
about this is that any two of the walkers will take an expected time
$\sim n/2$ to meet and coalesce; the fact that we are dealing with
an {\em unbounded} number of particles only increases the expected time by a constant factor.

It is natural to ask what happens in more general graphs. This is closely related to the following problem, which was posed by Aldous and Fill in the mid-nineties.

\begin{problem}[Open problem 13, Chapter 14 of \cite{AldousFill_RWBook}]\label{problem:main} Prove that there exists a universal constant $K>0$ such that the expected value of $C$ satisfies
$$\Ex{C}\leq K\,\Thit^{\bG}$$
irrespective of initial conditions, where $\Thit^{\bG}$ is the
maximum expected hitting time of a vertex in $\bG$.\end{problem}

To see how this relates to our previous discussion, consider a vertex-transitive graph $\bG$. Proposition $5$ in
\cite[Chapter 14]{AldousFill_RWBook} implies that the maximum
expected meeting time of two walkers on $\bG$, denoted by ${\rm
T}_{\rm meet}^{\bG}$, actually equals $\Thit^{\bG}/2$. This implies that, if Problem 1 has a positive solution, all vertex-transitive graphs are like $K_n$ in that $\Ex{C}$ is at most a universal constant factor away from ${\rm T}_{\rm meet}^\bG$. A similar conclusion holds for the many other families of
graphs where ${\rm T}_{\rm meet}^{\bG}=\theita{\Thit^\bG}$ (eg. all
regular graphs with $\Thit^\bG=\bigoh{n}$). For more general graphs it is still true
that ${\rm T}_{\rm meet}^{\bG}\leq {\rm T}_{\rm hit}^{\bG}$, as
proven in the aforementioned Proposition (see also
\cite{Aldous_MeetingTimes}), and the Problem may be viewed as an
strengthening of this fact\footnote{There are graphs such as stars where $\Thit^\bG$ is much larger than
${\rm T}_{\rm meet}^{\bG}$ or $\Ex{C}$.}.

To the best of our knowledge, Problem 1 has remained open up to now.
The best known bound of this sort has an extra $\ln|\bV|$ factor; see Proposition $11$ in
\cite[Chapter 14]{AldousFill_RWBook}. Our main goal in this paper is to give a solution of Problem 1 in
the more general setting of reversible Markov chains.

Assume that
$Q$ is the generator of a reversible, irreducible, continuous-time
Markov chain over a finite set $\bV$. Given $v\in \bV$, let $H_v$ be
the hitting time of $v$, ie. the first time at which a trajectory of
$Q$ hits $v$. We define the following parameter of the chain:
\begin{equation}\label{eq:thit_Q}\Thit^Q\equiv
\max_{v,w\in\bV}\Exp{w}{H_v} = \mbox{largest expected hitting time
for $Q$}.\end{equation}

Define a system of coalescing random walks as in the case of graphs,
with the difference that each walker now evolves over $\bV$
according to $Q$. The following Theorem solves Problem 1.
\begin{theorem}\label{thm:main}There exists a universal constant $K>0$ such that, with $Q$ as above, for any $n\in\N\backslash\{0\}$ and for any $x^{(n)}=(x(1),\dots x(n))\in\bV^n$:
$$\Exp{x^{(n)}}{C}\leq K\,\Thit^Q.$$\end{theorem}
\begin{remark}Here $x^{(n)}$ is an initial condition, with $n$ arbitrary. In particular, there may be more or less than one walker at each site $v\in\bV$ in the beginning of the process. Allowing for arbitrary initial conditions is convenient for our proofs, but does not really change the results.\end{remark}
We also prove a stronger result. Let $C_k$ denote the first time at
which there are at most $k$ clusters of coalesced walkers ($k\geq
1$). Notice that $C_1=C$ with this definition.
\begin{theorem}\label{thm:Cr}There exists a universal constant $K_1>0$ such that, in the same setting of \thmref{main}:
$$\forall k\in\N\backslash\{0\},\, \Exp{x^{(n)}}{C_k}\leq K_1\,\left(\frac{\Thit^Q}{k} +\Tmix^Q\right),$$
where $\Tmix^Q$ is the mixing time of $Q$ (see \secref{markovbasics} for a
definition).\end{theorem}

The dependence on $k$ in this Theorem is essentially best possible, as $\Exp{x^{(n)}}{C_k}\sim \Thit^Q/k$
for large complete graphs. The case $k=1$ gives back \thmref{main}, as $\Tmix^Q\leq c\,\Thit^Q$ for some universal $c>0$ \cite[Chapter 4]{AldousFill_RWBook}. We will nevertheless prove \thmref{main} first and then show how its proof can be modified to obtain \thmref{Cr}.

One justification for proving this second result is that it is helpful in approximating the {\em distribution} of $C$. We are in the process of writing a paper where we show that, if $Q$ is transitive and
$\Tmix^Q\ll \Thit^Q$, then
$$\mbox{Law of }\frac{C}{\Thit^Q}\approx \sum_{i=1}^{+\infty}\frac{Z_i}{i(i+1)}\mbox{, as for the complete graph (cf. \eqnref{distcomplete})}.$$
In particular, $\Ex{C}\sim \Thit^Q$. This was previously known only for discrete tori $\Z_L^d$ with $L\gg 1$ in $d\geq 2$ dimensions, due to Cox's paper \cite{Cox_Coalescing}\footnote{Transitivity can be dropped at
the cost of making stronger assumptions on $Q$ and using a different normalization factor.}. An important step in both our proof and Cox's argument is that $\Ex{C_k}\ll \Thit^Q$ if $k\gg 1$. Cox
proves this in \cite[Section 4]{Cox_Coalescing} via a simple
renormalization argument which is very specific for discrete tori, whereas we use \thmref{Cr} for the same purpose.

\subsection{Application to the voter model}\label{sec:voter}
We now sketch the connection between our results and the voter model \cite{Liggett_IPSBook,AldousFill_RWBook}
on a graph $\bG$ (this could be generalized to an arbitrary
generator $Q$, but we will not do this here). The
state of the process at a given time $t$ is a function:
$$\eta_t:V(\bG)\to \sO$$
where $V(\bG)$ is the vertex set of $\bG$ and $\sO$ is a fixed set of
possible {\em opinions}. The evolution of the process is as follows.
Each vertex $v\in V(\bG)$ ``wakes up" at rate $1$; when that happens at a time $t>0$, $v$ chooses one of its neighbors $w$
uniformly at random and updates its value of
$\eta_t(v)$ to $w$'s opinion $\eta_{t_-}(w)$; all other opinions stay the same.

A classical duality result (see eg. \cite[Chapter
5]{Liggett_IPSBook} or \cite[Chapter 14]{AldousFill_RWBook}) relates
the state of the process at a given time to a system of coalescing
random walks on $\bG$ moving backwards in time. In particular, the
consensus time for the voter model -- ie. the least time at which
all vertices of $\bG$ have the same opinion -- is dominated by the
coalescence time $C$ from the initial state with all vertices
occupied. This implies the following Corollary of \thmref{main}.

\begin{corollary}\label{cor:main}There exists a universal constant $K>0$ such that, for any graph $\bG$ and any set $\sO$, the expected value of the consensus time
of the voter model defined in terms of $\bG$ and $\sO$, started from
an arbitrary initial state, is bounded by $K\,\Thit^{\bG}.$\end{corollary}

Proposition $5$ in \cite[Chapter 14]{AldousFill_RWBook} shows that
the Corollary is tight up to the value of $K$ for vertex-transitive
$\bG$, at least when the initial conditions are iid uniform over
$\{-1,+1\}$ (say); we omit the details.

\subsection{Main proof ideas}\label{sec:outline}

Let us give an outline of the (elementary) proof of \thmref{main};
the proof of \thmref{Cr} is quite similar. For clarity, we first present an oversimplified account, and then explain how one can avoid the oversimplifications.

We label the $n$ walkers $(X_t(a))_{t\geq 0}$ with numbers $a=1,\dots,n$. Instead of
having walkers coalesce, we will assume that a walker $\# b$ will
kill any walker $\# a$ with $a>b$ that happens to be in the same
state as itself (this is made precise in \secref{Yprocess}). The number of walkers that are alive at time $t$ in
this process is precisely the number of clusters in the coalescing random walks process, and $C$ is the first time at which
only walker $\# 1$ is still alive. This implies that:
$$\Pr{C>t}\leq \sum_{a=2}^n \Pr{\mbox{walker \# $a$ alive at time
$t$}}.$$
We now make the following oversimplification:
\begin{quotation}\noindent {\bf Oversimplification $\# 1$:} walker $\# a$ dies at the first time when $X_t(a)=X_t(b)$ for some $b<a$.\end{quotation}

The reason why this is an oversimplification is that a walker $\# b$
may have died before meeting walker $\# a$. For the moment, we ignore
this and write:
$$\Pr{\mbox{walker \# $a$ alive at time $t$}}\leq \Pr{\bigcap_{b=1}^{a-1}\{\forall 0\leq s\leq t,\,X_s(a)\neq
X_s(b)\}}.$$ In order to simplify the RHS, we notice that the
trajectories $(X_t(u))_{t\geq 0}$ of walkers $\# u$, $1\leq u\leq
a$, are independent realizations of $Q$. Conditioning on
$X_s(a)=h_s$, $s\geq 0$, makes the events in the RHS independent,
and we deduce:
$$\Pr{\mbox{walker \# $a$ alive at time $t$}\mid (X_s(a))_{s\geq 0}=(h_s)_{s\geq 0}}\leq \prod_{b=1}^{a-1}\Pr{\forall 0\leq s\leq t,\,X_s(b)\neq h_s}.$$

We now make another oversimplification.

\begin{quotation}\noindent {\bf Oversimplification $\# 2$:} $(X_t(b))_{t\geq 0}$ is started from the stationary distribution for all $b$.\end{quotation}

This allows us to use the following Lemma, which we believe to be
new (and of independent interest).

\begin{lemma}[Meeting time Lemma; proven in \secref{proof_meetingtime}]\label{lem:meetingtime}Let $(X_t)_{t\geq 0}$ be a realization of $Q$ starting from the stationary distribution $\pi$. Then there exist $v\in\bV$ and a
quasistationary distribution ${\qstat}_v$ for $\bV\backslash\{v\}$
such that for any deterministic path $h\in\bbD([0,+\infty),\bV)$, we
have:
$$\forall t\geq 0,\, \Prp{\pi}{\forall 0\leq s\leq t,\, X_t\neq h_t}\leq \Prp{{\qstat}_v}{H_v> t}=\exp\left(-\frac{t}{\Exp{{\qstat}_v}{H_v}}\right).$$\end{lemma}
\begin{remark}The proof of \lemref{meetingtime} shows that we may take $v\in h([0,+\infty))$. This is a well-known result if $h\equiv v$ \cite[Chapter 3, Section 6.5]{AldousFill_RWBook}. An application of this Lemma to so-called cat-and-mouse games is sketched in the final section.\end{remark}

Notice that $\Exp{{\qstat}_v}{H_v}\leq \Thit^Q$, so:
$$\Pr{\mbox{walker \# $a$ alive at time $t$}\mid (X_s(i))_{s\geq 0}=(h_s)_{s\geq 0}}\leq
e^{-\frac{(a-1)t}{\Thit^Q}}.$$ This shows that:
$$\Pr{C>t}\leq \sum_{a=2}^n e^{-\frac{(a-1)t}{\Thit^Q}}.$$
If one takes $t = (\ln 2 + c)\Thit^Q$, the RHS becomes:
$$\Pr{C>(\ln 2 + c)\Thit^Q}\leq \sum_{a=2}^n 2^{-a+1}e^{-(a-1)c}\leq e^{-c},$$
and this gives $\Ex{C}\leq (\ln 2 + 2)\Thit^Q$.

Of course, this is {\em not} a proof of \thmref{main} because of the
oversimplifications. Our way out of this is to introduce a process
where at any given time there is a list of {\em allowed killings}.
At any time $t$ there will be a set $\sA_t$, so that walker $\# b$
may kill walker $\# a$ at time $t$ only if $b<a$ and $(b,a)\in \sA_t$ (cf. \secref{allowed}). The salient characteristics of this
process are:
\begin{enumerate}
\item For any choice of $\sA=(\sA_t)_{t\geq 0}$, the set of alive
walkers in the process defined via $\sA$ dominates the corresponding
set in the process without $\sA$ (see \propref{domination}).
\item A judicious choice of $\sA$ will ensure that for each $a$, there will
be a large enough time interval where a large number of walkers will
be available to kill walker $\# a$. Moreover, many of these will be
stationary.\end{enumerate}

Item $1$ allows us to consider the process with a list of allowed killings instead of the original process in order to obtain upper bounds. Item $2$ will mean that we may apply the Meeting Time Lemma to at least some of the walkers with indices $b<a$, in some time intervals. These two ingredients will allow us to ``fix" the oversimplified proof just presented.

\subsection{Organization}

The remainder of the paper is organized as follows. \secref{prelim}
introduces our notation and recalls some basic concepts. \secref{mainprocesses} defines the main processes we consider in the paper. \secref{mainproofs}
presents the proofs of the two Theorems, and
\secref{meetingtime} presents the proof of
\lemref{meetingtime}. Some final comments are presented in the last Section.

\section{Preliminaries}\label{sec:prelim}

In what follows we recall some basic material while also fixing
notation.

\subsection{Basic notation}
$\N$ is the set of non-negative integers. Given
$n\in\N\backslash\{0\}$, we set $[n]\equiv\{1,2,\dots,n\}$.

We will often speak of {\em universal constants}. These are numbers
that are independent of any other object or parameter under
consideration, be it a Markov chain, the initial state of a process
under consideration or anything else.

The
cardinality of a finite set $S$ is denoted by $|S|$, and $2^S$ represents the power set of $S$ (ie. the set whose elements are the subsets of $S$). The set of all probability
measures over $S$ will be denoted by $M_1(S)$. $\R^S$ denotes the space of all functions $f:S\to\R$, or equivalently of all (column) vectors with entries indexed by $S$. Linear operators acting on $\R^{S}$ correspond to matrices with rows and columns indexed by the elements of $S$. If $A$ is some matrix of this sort, $\|A\|_{\rm op}$ is the operator norm of $A$. If $A$ is symmetric, we let $\lambda_{\min}(A)$, $\lambda_{\max}(A)$ denote the minimum and maximum eigenvalues of $A$ (respectively).

Given a finite set $F\neq \emptyset$, a function $\omega:[0,+\infty)\to F$ is said to be c\`{a}dl\`{a}g if there exist
$t_0=0<t_1<t_2<\dots<t_n<\dots\nearrow +\infty$ with $\omega$ constant
over each interval $[t_{i-1},t_{i})$. $\bbD([0,+\infty),F)$ is the set of all such c\`{a}dl\`{a}g functions, with the $\sigma$-field generated by the projections ``$\omega\mapsto \omega(t)$" ($t\geq 0$).

\subsection{Markov chain basics}\label{sec:markovbasics}

Let $\bV$ be a finite, non-empty set. A matrix $Q$ (with rows and columns labelled by
$\bV$) which acts on $\R^{\bV}$ in the following way:
$$Q:f(\cdot)\in\R^{\bV}\mapsto
\sum_{x\in\bV, x\neq \cdot}q(\cdot,x)(f(\cdot)-f(x)),\mbox{ with
}q(\cdot,\cdot\cdot)\geq 0$$ defines a unique {\em continuous-time
Markov chain} on $\bV$. More precisely, there exists a unique family of measures $\{\Prpwo{x}\}_{x\in\bV}$
over $\bbD([0,+\infty),\bV)$ (with the $\sigma$-field generated by finite-dimensional
projections) such that, letting
$$X_t:\omega\in \bbD([0,+\infty),\bV)\mapsto X_t(\omega)\equiv \omega(t)\in\bV$$
and $\sF_t\equiv \sigma(X_s\,:\,s\leq t)$, we have
$\Prp{x}{X_0=x}=1$ and
\begin{equation}\label{eq:transitionprobs}\Prp{x}{X_{t+s}=y\mid\sF_s} = \mbox{ the entry of $e^{-tQ}$ labelled by $(X_t,y)$} \,(t,s\geq 0,y\in\bV).\end{equation}
$Q$ is said to be the generator of the Markov chain and
the numbers $q(x,y)$ ($x,y\in\bV$, $x\neq y$) are the transition rates. We let
$\Exp{x}{\cdot}$ denote expectation with respect to $\Prpwo{x}$.

We also define $$\Prpwo{\mu} =
\sum_{x\in\bV}\mu(x)\Prpwo{x},\;\;(\mu\in M_1(\bV))$$ which
we interpret in the customary way, as describing the law of the chain given by $Q$ from a random
initial state with law $\mu$. $\Exp{\mu}{\cdot}$ is the corresponding expectation symbol.

We will always assume that $Q$ is {\em irreducible}, meaning that for all $A\subset
\bV$ with $A,\bV\backslash A\neq\emptyset$ there exist $a\in A$ and
$b\in\bV\backslash A$ with $q(a,b)\neq 0$. In this case there exists
a unique probability measure $\pi\in M_1(\bV)$ which is {\em
stationary} in the sense that $\Prp{\pi}{X_t=\cdot}=\pi(\cdot)$ for
all $t\geq 0$. Moreover, we have that:
$$\forall x,y\in\bV,\, \lim_{t\to +\infty}\Prp{x}{X_t=y} = \pi(y).$$
The {\em mixing time} of $Q$ measures the speed of this convergence:
$$\Tmix^Q\equiv \inf\left\{t\geq 0\,:\, \forall x\in\bV,\, \max_{S\subset \bV}|\Prp{x}{X_t\in S}-\pi(S)|\leq 1/4\right\}.$$

Finally, we will also assume that $Q$ is {\em reversible} with
respect to $\pi$, which means that $\pi(x)q(x,y)=\pi(y)q(y,x)$ for
all distinct $x,y\in\bV$. This is the same as requiring that the
matrix $\Pi^{1/2}\,Q\,\Pi^{-1/2}$ is symmetric, where $\Pi$ is
diagonal and has the values $\pi(v)$, $v\in\bV$ on the diagonal.

\section{Processes with multiple random walks}\label{sec:mainprocesses}

We define here the main processes that we will be concerned with, all of which involve $n$ random walkers for some integer $n\in\N\backslash\{0,1\}$. We will assume that $Q$ and $\{\Prpwo{x}\}_{x\in \bV}$ are as defined in \secref{markovbasics}.

\subsection{Independent random walks}
We first define a processes made out of $n$ independent realizations of
the Markov chain with generator $Q$. More specifically, given
$$x^{(n)}=(x(1),x(2),\dots,x(n))\in\bV^n,$$
we let $\Prpwo{x^{(n)}}$ denote the distribution on
$\bbD([0,+\infty),\bV^n)$ corresponding to $n$ independent
trajectories of $Q$, \begin{equation}\label{eq:introduceindQ}(X_t^{(n)})_{t\geq 0}\equiv (X_t(a)
\,: \,a\in[n])_{t\geq 0},\end{equation} with each $(X_t(a))_{t\geq 0}$
started from $x(a)$. That is, the joint law of $\{(X_t(a))_{t\geq 0}\}_{a\in [n]}$ is the product measure:
$$\Prpwo{x(1)}\times \Prpwo{x(2)}\times \dots \times \Prpwo{x(n)}.$$
Notice that our notation $\Prpwo{x^{(n)}}$ does not refer explicitly
to the fact that this is a process on $\bV^n$, as opposed to the
process over $\bV$ defined in the previous subsection. This
distinction should be clear from context and from the fact that we
write all $x^{(n)}\in\bV^n$ with a ``$(n)$" superscript. The independent random walks process is also a Markov chain: for $x^{(n)}=(x(1),\dots,x(n))$ and $y^{(n)}=(y(1),\dots,y(n))$ distinct, the transition rate from $x^{(n)}$ to $y^{(n)}$ is:
\begin{equation}\label{eq:transitionsindep}q^{(n)}(x^{(n)},y^{(n)})\equiv \left\{\begin{array}{ll}q(x(i),y(i))&\mbox{if }x(i)\neq y(i)\mbox{ and }x(j)=y(j)\mbox{ for all }j\in[n]\backslash\{i\};\\ 0& \mbox{otherwise}.\end{array} \right.\end{equation}

\subsection{Coalescing random walks}\label{sec:CRW} For our purposes, it is convenient to define this process, denoted by
$$(\Co_t^{(n)})_{t\geq 0}\equiv (\Co_t(a) \,:\, a\in[n])_{t\geq 0}$$
as a deterministic function of the independent random walks process. The idea is that, once a walker meets
another walker with smaller index, it starts following the
trajectory of the latter. That is, consider a realization of
$\Prpwo{x^{(n)}}$ as in \eqnref{introduceindQ}.  First define:
$$\Co_t(1)\equiv X_t(1)\;\;(t\geq 0).$$Given
$a\in[n]\backslash\{1\}$, assume inductively that $(\Co_t(b))_{t\geq
0}$ has been defined for $1\leq b<a$. Since $Q$ is irreducible,
there a.s. is a {\em first} time $\tau_a$ at which $X_t(a)=\Co_t(b)$
for some $b\in [a-1]$. More precisely, define:
$$\tau_a\equiv \inf\{t\geq 0\,:\, \exists 1\leq
b<a,\,X_t(a)=\Co_t(b)\}$$ and
$$B_a\equiv \min\{b\in[a-1]\,:\, X_{\tau_a}(a)=X_{\tau_a}(b)\}$$ and then set:
$$\Co_t(a)\equiv \left\{\begin{array}{ll}X_t(a), & 0\leq t<\tau_a;\\ \Co_t(B_a), & t\geq \tau_a;  \end{array}\right.\;\mbox{ for each }t\geq 0.$$
One can show that the law of $(\Co_t^{(n)})_{t\geq 0}$ is invariant under permutations of the $x(i)$. We also {\em define} the set:
\begin{equation}\label{eq:defSt}S_t\equiv \{v\in\bV\,:\, \exists a\in[n],\, \Co_t(a)=v\}\end{equation}
as the set of occupied sites in this process. Our definition of
$(S_t)_{t\geq 0}$ coincides with the more traditional coalescing
random walks process defined in eg. \cite{Cox_Coalescing}. We also
set:
$$C_k\equiv \inf\{t\geq 0\,: \,|S_t|\leq k\}\;\;(k\in\N\backslash\{0\})$$
and $C\equiv C_k$.
\begin{remark}\label{rem:repeats}We note that this process makes sense even if $x^{(n)}$ contains repeats, ie. if there exist $i\neq j$ with $x(i)=x(j)$.\end{remark}

\subsection{Random walks with killings}\label{sec:Yprocess} Let $\partial\not\in V$ be a
``coffin state". We define a new process
$$(Y^{(n)}_t)_{t\geq 0}\equiv(Y_t(a) \,:\, a\in[n])_{t\geq 0}.$$The new idea is that a walker with index $a$ will be
killed by a walker of index $b<a$ occupying the same site. More precisely, we first define:
$$Y_t(1)\equiv X_t(1)\;\;(t\geq 0).$$Given
$a\in[n]\backslash\{1\}$, assume inductively that $(Y_t(b))_{t\geq
0}$ has been defined for $1\leq b<a$. Define:
$$\tau_a\equiv \inf\{t\geq 0\,:\, \exists 1\leq
b<a,\,X_t(a)=Y_t(b)\}$$ and set:
$$Y_t(a)\equiv \left\{\begin{array}{ll}X_t(a), & 0\leq t<\tau_a;\\ \partial, & t\geq \tau_a;  \end{array}\right.\;\mbox{ for each }t\geq 0.$$
Although our new definition of $\tau_a$  different from the previous
one, it is easy to show that the two definitions coincide, and that
in fact:
\begin{proposition}[Proof omitted] Let $S_t$ be as \eqnref{defSt}. Then for all $t\geq 0$,
$$S_t=\{v\in \bV\,:\, \exists a\in [n],\, Y_t(v)=a\}$$
and
$$|S_t| = |\{a\in [n]\,:\,Y_t(a)\neq \partial\}|.$$
Therefore, for all $k\in\N\backslash\{0\}$,
$$\Prp{x^{(n)}}{C_k>t} = \Prp{x^{(n)}}{|S_t|\geq k+1} = \Prp{x^{(n)}}{|\{a\in[n]\,:Y_t(a)\neq \partial\}|\geq k+1}.$$\end{proposition}
\begin{remark}\label{rem:repeats2}As in \remref{repeats}, we may allow $x^{(n)}$ where $x(i)=x(j)$ for some pair $i\neq j$. Notice, however, that $Y_0^{(n)}\neq x^{(n)}$ in this case.\end{remark}

\subsection{Random walks with a list of allowed killings}\label{sec:allowed} Now assume that we have
a deterministic c\`{a}dl\`{a}g trajectory:
$$\sA:t\geq 0\mapsto 2^{[n]^2}.$$
We define yet another process:$$((Y^{\sA}_t)^{(n)})_{t\geq 0}\equiv (Y^\sA_t(a) \,:\, a\in[n])_{t\geq 0}$$
where a walker with index $a$ may be killed by a walker with index
$b$ only if they occupy the same site at some time $t$ {\em and}
$(b,a)\in\sA_t$. Intuitively, this means that $b$ is allowed to kill $a$ only at times $t$ with $(b,a)\in \sA_t$.

For a formal definition, we first set:
$$Y^\sA_t(1)\equiv X_t(1)\;\;(t\geq 0).$$Given
$a\in[n]\backslash\{1\}$, assume inductively that
$(Y^\sA_t(b))_{t\geq 0}$ has been defined for $1\leq b<a$. Define:
$$\tau^\sA_a\equiv \inf\{t\geq 0\,:\, \exists 1\leq
b<a,\,(b,a)\in\sA_t\mbox{ and }X_t(a)=Y^{\sA}_t(b)\}$$ and set:
$$Y^{\sA}_t(a)\equiv \left\{\begin{array}{ll}X_t(a), & 0\leq t<\tau^{\sA}_a;\\ \partial, & t\geq \tau^{\sA}_a;  \end{array}\right.\;\mbox{ for each }t\geq 0.$$
The following Proposition shows that the process with a list of allowed killings can be used to upper bound $\Exp{x^{(n)}}{C_k}$.
\begin{proposition}\label{prop:domination}Define:
$$S^{\sA}_t\equiv \{Y_t^\sA(a)\,:\, a\in[n]\}.$$
For any choice of $\sA$ as above and of initial state $x^{(n)}$, one can couple $(S_t)_{t\geq 0}$ and $(S^{\sA}_t)_{t\geq 0}$ such that (almost surely) $S_t\subset S^\sA_t$ for all $t\geq 0$. In particular, for all $k\in\N\backslash\{0\}$,
$$\Prp{x^{(n)}}{C_k>t}=\Prp{x^{(n)}}{|S_t|\geq k+1}\leq \Prp{x^{(n)}}{|S^{\sA}_t|\geq k+1}.$$\end{proposition}
We omit the proof of this rather intuitive Proposition. The key idea here is this: suppose we do {\em not} kill a walker $a$ at a given time $t_0$. The only way this could make $S_t$ ``smaller" is if $X_t(a)$ were to meet a walker $X_t(c)$ with $c>a$ at some later time $t\geq t_0$. But if this happens, we may pretend that $X_s(a)$ follows the trajectory of $X_s(c)$ for $s\geq t$; this follows from the Markov property coupled with the fact that $X_t(a)=X_t(c)$. This shows that in fact $S_t$ does {\em not} become smaller.
\begin{remark}\label{rem:repeats3}Similarly to \remref{repeats2}, we note that we may allow $x^{(n)}$ with $x(i)=x(j)$ for some pair $i\neq j$, but then $(Y_0^\sA)^{(n)}\neq x^{(n)}$.\end{remark}
\section{Proofs of the main Theorems}\label{sec:mainproofs}

We prove Theorems \ref{thm:main} and \ref{thm:Cr} in this Section. For simplicity, we first focus on the proof of \thmref{main}, and then show how it can be modified to prove the second Theorem. We will take the notation and definitions in Sections \ref{sec:markovbasics} and \ref{sec:mainprocesses} for granted.

\subsection{Preliminaries for the proof of \thmref{main}}

We first note that \thmref{main} follows from a seemingly different statement.

\begin{proposition}\label{prop:simpler}Let $c,\gamma>0$ be given universal constants. Suppose we can show that there exists some choice of $\sA=(\sA_t)_{t\geq 0}$ as in \secref{allowed} and of $0\leq t_0\leq c\,(\Tmix^Q + \Thit^Q)$ with
$$\forall n\in\N\backslash\{0\},\, \forall x^{(n)}\in \bV^n,\, \Prp{x^{(n)}}{|S^\sA_{t_0}|\geq 2}\leq 1-\gamma.$$
Then
$$\forall n\in\N\backslash\{0\},\, \forall x^{(n)}\in \bV^n,\, \Exp{x^{(n)}}{C}\leq K\,\Thit^Q$$
where $K>0$ is universal.\end{proposition}
\begin{proof}Given $s\geq 0$, denote:
\begin{equation}\label{eq:defEs}E(s)\equiv \{C>s\} = \{|S_s|\geq 2\} = \bigcup_{a\in [n]\backslash\{1\}}\{\tau_a>s\}.\end{equation}
Combining the assumption of the Proposition with \propref{domination} gives:
\begin{equation}\label{eq:oneeventhere}\forall n\in\N\backslash\{1\},\, \forall x^{(n)}\in \bV^n,\, \Prp{x^{(n)}}{E(t_0)} \leq  \Prp{x^{(n)}}{|S^\sA_{t_0}|\geq 2}\leq 1-\gamma\end{equation}
We now consider $E(\ell t_0)$ where $\ell>1$ is an integer. Let $(\Theta_{s})_{s\geq 0}$ denote the time-shift operators for the independent random walks process and let $(\sF^{(n)}_{s})_{s\geq 0}$ denote the filtration generated by this process.
\begin{eqnarray*}\Prp{x^{(n)}}{E(kt_0)}&\leq & \Prp{x^{(n)}}{E((\ell-1)t_0)\cap \left(\cup_{a=2}^n\{\tau_a\circ \Theta_{(\ell-1)t_0}>t_0\}\right)}
\\ &=& \Prp{x^{(n)}}{E((\ell-1)t_0)\cap \Theta^{-1}_{(\ell-1)t_0}(E(t_0))}\\
\mbox{($E((\ell-1)t_0)\in\sF^{(n)}_{(\ell-1)t_0}$)}
&\leq &\Exp{x^{(n)}}{\Ind{E((\ell-1)t_0)}\Prp{x^{(n)}}{\Theta_{(\ell-1)t_0}^{-1}(E(t_0))\mid\sF^{(n)}_{(\ell-1)t_0}}}\\
\mbox{(Markov property)}&=&
\Exp{x^{(n)}}{\Ind{E((\ell-1)t_0)}\Prp{X^{(n)}_{(\ell-1)t_0}}{E(t_0)}}\\
\mbox{(inequality \eqnref{oneeventhere})}&\leq
&\Exp{x^{(n)}}{\Ind{E((\ell-1)t_0)}\,(1-\gamma)}\\ &=&
\Prp{x^{(n)}}{E((\ell-1)t_0)}\,(1-\gamma) \\ \mbox{(induction on
$k$)}&\leq & (1-\gamma)^\ell.\end{eqnarray*} Recalling the
definition of $E(s)$, we deduce that:
\begin{equation*}\label{eq:AAAA}\forall n\in\N\backslash\{1\},\, \forall x^{(n)}\in \bV^n,\, \Exp{x^{(n)}}{C} \leq  \sum_{\ell\in\N}(1-\gamma)^\ell t_0=\frac{t_0}{\gamma}\leq \frac{c\,(\Tmix^Q + \Thit^Q)}{\gamma}.\end{equation*}The Proposition follows from this because $\Tmix^Q\leq c_0\,\Thit^Q$ for some universal $c_0>0$ \cite[Chapter 3]{AldousFill_RWBook} and both $c$ and $\gamma$ are universal.\end{proof}

\subsection{Construction of $\sA$}\label{sec:epochs}
\begin{notation}\label{not:convention}From now on, we fix some $x^{(n)}$ and write $\Prwo$ instead of $\Prpwo{x^{(n)}}$.\end{notation} We will now design a specific trajectory $\sA=(\sA_t)_{t\geq 0}$ which will allow for a simple analysis of $S^{\sA}_t$. Let $m\in\N$ be the smallest non-negative
number with $n\leq \sum_{i=0}^m2^i$. Define sets
\begin{eqnarray*}A_0&=&\{1\};\\ A_j&\equiv&
\left[\sum_{i=0}^j2^i\right]\backslash \left[\sum_{i=0}^{j-1} 2^i\right]\; (1\leq j\leq
m-1);\\ \mbox{and }A_m &\equiv &[n]\backslash \left[\sum_{i=0}^{m-1}2^i\right].\end{eqnarray*}
We will consider different {\em epochs}, numbered backwards in time.
It is convenient to have the following notation.
\begin{equation}\label{eq:deftj}\begin{array}{llll}t_m&\equiv&2\Tmix^Q; &\\ t_j &=& t_{j+1}+(\ln
5)\,2^{4-j}\,\Thit^Q, & j=m-1,m-2,\dots,0.\end{array}\end{equation}

\begin{enumerate}
\item {\em Epoch $\# \infty$} is the time interval $[0,t_m)$. We set $\sA_t\equiv \emptyset$ for all $t$ in this interval, ie. no killings are allowed up to time $2\Tmix^Q$.
\item {\em Epochs $\# m$ through $\# 1$} correspond to time intervals $I_j=[t_j,t_{j-1})$ as $j$ decreases from $m$ to $1$. For each such $j$ we set:
$$\sA_t \equiv A_{j-1}\times \cup_{p=j}^m A_{j},\, t\in I_j.$$
That is, the only killings allowed are between walkers with
labels in $A_{j-1}$ and $A_{p}$ with $p\geq j$.
\item {\em Epoch $\# 0$} corresponds to the time interval,
$$I_0\equiv \left[t_0,+\infty\right),$$
(the remaining time), where we set $\sA_t\equiv
[n]^2$.\end{enumerate}

We note for later convenience that:
\begin{equation}\label{eq:boundt_0}t_0 \leq 2\Tmix^Q + (\ln 5)\sum_{j\geq 0}\,2^{4-j}\,\Thit^Q\leq c\,(\Tmix^Q + \Thit^Q)\end{equation}
with $c>0$ universal, since $\sum_j2^{-j}<+\infty$. We will use this in our application of \propref{simpler}.

\subsection{Abundance of good walkers}\label{sec:goodwalkers}

We have the following simple proposition about the epoch $\#\infty$.
Intuitively, it says that, at the end of this epoch, a positive proportion of the random
walkers are ``good", in that they have converged to stationarity.

\begin{proposition}\label{prop:goodwalkers}One can construct a (random) subset $R\subset [n]$ such that:
\begin{enumerate}
\item $R$ is $\sH^{(n)}_{t_m}$-measurable, where $\sH^{(n)}_{t_m}$ is the sigma-field generated by $(X^{(n)}_s)_{s\leq t_m}$ and by some additional independent random variable $U$.
\item Each $r\in [n]$ belongs to $R$ with probability $1/4$,
independently of all other $r'\in [n]$.
\item Conditionally on $R$ and on $(X_{t_m}(i))_{i\in[n]\backslash
R}$, the vector $(X_{t_m}(r))_{r\in R}$ has iid coordinates, each with distribution $\pi$.
\end{enumerate}\end{proposition}
\begin{proof}Consider a single $a\in[n]$. Since $Q$ is reversible, Lemma 7 in \cite[Chapter 4]{AldousFill_RWBook} shows that:
$$\forall a\in[n],\,\forall v\in\bV\,:\,
\Pr{X_{2\Tmix^Q}(a)=v}\geq \frac{\pi(v)}{4};$$ in other
words, for each $a$ there exists some $\nu_a\in M_1(\bV)$ such that:
$$\Pr{X_{2\Tmix^Q}(a)=\cdot} = \frac{1}{4}\,\pi(\cdot) +
\frac{3}{4}\,\nu_a(\cdot).$$ Since the random variables $(X_{t_m}(a))_{a\in [n]}$ are independent, we may assume that
they sampled as follows:
\begin{enumerate}
\item Let $(I(a))_{a\in[n]\backslash A_m}$ be iid with
$\Pr{I(a)=1}=1-\Pr{I(a)=0}=1/4$.
\item For each $a$ with $I(a)=1$, let $X_{t_m}(a)$ be a sample
from $\pi$, independent of everything else.
\item For each $b$ with $I(b)=0$, let $X_{t_m}(b)$ be a sample
from $\nu_b$, independent of everything else.
\end{enumerate}
One may check that $R\equiv \{a\in[n]\backslash A_m\,:\,I(a)=1\}$
has the desired properties.\end{proof}

The next proposition means that, with positive
probability, there is a constant proportion of good walkers within
each $A_i$ with $i\leq m-1$.

\begin{proposition}\label{prop:abundance}Let $\sG$ be the event:
$$\sG\equiv \bigcap_{i=0}^{m-1}\{|R\cap A_i|\geq 2^{i-3}\}.$$
Then $\Pr{\sG}\geq \alpha>0$, where $$\alpha\equiv
\prod_{i=0}^{+\infty}(1-e^{-2^{i-7}})>0$$ is
universal.\end{proposition}
\begin{proof}Let ${\rm Bin}(m,x)$ denote a binomial random variable with parameters $m$ and $x$, so that:
$$\Pr{{\rm Bin}(m,x)=k} = \binom{m}{k}x^k(1-x)^{m-k}\;\; (k\in
[m]\cup\{0\}).$$ The random variables $N_i=|R\cap A_i|$, $0\leq
i\leq m-1$ are independent, and each $N_i$ has the law of ${\rm
Bin}(2^i,1/4)$. Chernoff bounds \cite[Appendix A.1]{AlonSpencer_Method} imply:
$$\Pr{|R\cap A_i|<2^{i-3}}=\Pr{{\rm Bin}(2^i,1/4)<\Ex{{\rm
Bin}(2^i,1/4)} - 2^{i-3}}\leq e^{-\frac{(2^{i-3})^2}{2\,.\,2^i}}=
e^{-2^{i-7}}.$$ We deduce:
$$\Pr{\sG}\geq \prod_{i=0}^{m-1}\Pr{{\rm
Bin}(2^i,1/4)<2^{i-3}}\geq \alpha.$$ The positivity of $\alpha$
follows from $0<e^{-2^{i-7}}<1$ for all $i$ and
$\sum_ie^{-2^{i-7}}<+\infty$.\end{proof}

\subsection{The probability of being alive}

Let $E(a,t)$ denote the event:
$$E(a,t)\equiv \{Y^\sA_t(a)\neq \partial\} = \{\tau^\sA_a>t\}.$$
Notice that: \begin{equation}\label{eq:uppereasy}\{|S^\sA_t|\geq
2\}\subset \bigcup_{a=2}^nE(a,t).\end{equation}
We will now compute estimate
the conditional probability of $E(a,t)$ given $\sG$.

\begin{proposition}\label{prop:mainbound}Let $a\in A_i$ for some $1\leq i\leq m$. Then for all $1\leq j<i$:
$$\Pr{E(a,t_{j})\mid \sG}\leq 5^{j-i}.$$\end{proposition}
\begin{proof}We will prove a stronger statement: that for almost all $R_0\subset [n]$ and
$(h_t)_{t\in [t_i,t_{j})}$:
\begin{equation}\label{eq:forallR}\Pr{E(a,t_{j})\mid R=R_0,(X_t(a))_{t\in [t_i,t_{j})}=(h_t)_{t\in [t_i,t_{j})}}\leq 5^{-\sum_{r=j}^{i-1}\frac{|R_0\cap A_{r}|}{2^{r-4}}}.\end{equation}
This implies the proposition because the occurrence of $\sG$ implies
$|R\cap A_{r}|\geq 2^{r-4}$ for all $1\leq r<m-1$.

To prove \eqnref{forallR} we first observe that the event $E(a,t_{j})$ satisfies:
\begin{claim}Suppose $b\in A_{r}$ with $j\leq r<i$. Then:
$$E(a,t_{j})\subset\{\forall t\in [t_{r+1},t_{r}),\, X_t(a)\neq X_t(b)\}.$$\end{claim}
\begin{proof}[of the Claim] If the event in the RHS does not hold, there exists a $t\in [t_r,t_{r-1})$ with $X_t(a)=X_t(b)$.
We now argue that $\tau^\sA_a\leq t$ in this case. Indeed, this follows from the
definition of $\tau^\sA_a$ and the following observations:
\begin{enumerate}
\item $X_t(a)=Y^\sA_t(b)$: this follows from
$X_t(b)=Y^{\sA}_t(b)$, which is a consequence of the fact that
$(b,c)\not\in \sA_s$ for any $c>b$ and $s\leq t_{r}$ (ie. $b$ cannot
be killed before time $t_r$).
\item $(b,a)\in\sA_t$: this follows from $t\in
[t_{r+1},t_{r})=I_{r+1}$.
\end{enumerate}\end{proof}

The Claim implies:
\begin{multline*}\Pr{E(a,t_{j})\mid R=R_0,(X_t(a))_{t\in [t_i,t_{i-1})}=(h_t)_{t\in [t_i,t_{i-1})}}\\ \leq \Pr{\bigcap_{r=j}^{i-1}\bigcap_{b\in A_{r}}\{\forall t\in [t_{r},t_{r-1}),\, X_t(a)\neq X_t(b)\}\mid R=R_0,(X_t(a))_{t\in [t_i,t_{j})}=(h_t)_{t\in [t_i,t_{j})}}\\
\leq \Pr{\bigcap_{r=j}^{i-1}\bigcap_{b\in A_{r}\cap R_0}\{\forall
t\in [t_{r+1},t_{r}),\, X_t(a)\neq X_t(b)\}\mid R=R_0,(X_t(a))_{t\in
[t_i,t_{j})}=(h_t)_{t\in [t_i,t_{j})}}.\end{multline*} Now observe
that we are conditioning on $R=R_0$ and on the trajectory of
$(X_s(a))_{s\in [t_i,t_{i-1})}$. Since $a\not\in A_{i-1}$,
\propref{goodwalkers} implies that:
$$\mbox{Under the conditioning, $(X_{t_m}(b)\,:\,b\in R_0\cap \left(\cup_{r=j}^{i-1}A_r\right))$ are iid with common law $\pi$.}$$
Since $R$ is $\sH^{(n)}_{t_m}$-measurable, the Markov property for
the independent random walks process implies that $$\mbox{Under the
conditioning, $(X_{t+t_m}(b)\,:\,b\in R_0\cap \left(\cup_{r=j}^{i-1}A_r\right))_{t\geq
0}$ are iid realizations of $\Prpwo{\pi}$.}$$ We deduce:
\begin{multline*}\Pr{E(a,t_{j})\mid R=R_0,(X_t(a))_{t\in
[t_i,t_{j})}=(h_t)_{t\in [t_i,t_{j})}}\\ =
\prod_{r=j}^{i-1}\prod_{b\in R_0\cap A_{r}}\Prp{\pi}{\forall t\in
[t_{r+1},t_{r}),\, X_t\neq h_{t}}\\\mbox{($\Prpwo{\pi}$ is
stationary)} = \prod_{r=j}^{i-1}\Prp{\pi}{\forall t\in
[0,t_{r}-t_{r+1}),\, X_t(b)\neq h_{t+t_r}}^{|R_0\cap
A_r|}.\end{multline*} We apply the Meeting Time Lemma
(\lemref{meetingtime} above) to each term in the product and deduce
that, for some choice of $(v,{\qstat}_v)$ as in the Lemma,
\begin{multline*}\Pr{E(a,t_{i-1})\mid R=R_0,(X_t(a))_{t\in
[t_i,t_{i-1})}=(h_t)_{t\in [t_i,t_{i-1})}}\\ \leq
\prod_{r=j}^{i-1}\left\{\exp\left(-\frac{(t_{r}-t_{r+1})|R_0\cap
A_{r}|}{\Exp{{\qstat}_v}{H_v}}\right)\right\}.\end{multline*} The
proof of \eqnref{forallR} finishes once we realize that
$t_{r}-t_{r+1}= 2^{4-r}(\ln 5)\Thit^Q$ and $\Thit^Q\geq
\Exp{{\qstat}_v}{H_v}$.\end{proof}
\subsection{End of proof of \thmref{main}}

We now complete the proof of \thmref{main}. By \propref{simpler}, it suffices to show that:
$$\Prp{x^{(n)}}{|S^{\sA}_{t_0}|\geq 2}\leq 1-\gamma$$
for some universal $\gamma>0$, with $t_0$ as in \eqnref{boundt_0}.
To see this, we will use \eqnref{uppereasy} and recall our
convention of omitting $x^{(n)}$ from the notation (cf.
Notational convention \ref{not:convention}).
\begin{eqnarray*}\Pr{|S^{\sA}_{t_0}|\geq 2} &=& \Pr{\bigcup_{a=2}^nE(a,t_0)}\\
\mbox{($\sG$ as in Prop. \ref{prop:abundance})} &\leq & 1-\Pr{\sG} + \Pr{\sG\cap \bigcup_{a=2}^nE(a,t_0)}\\ \mbox{(union bound)} &\leq & 1-\Pr{\sG} + \sum_{a=2}^n\Pr{\sG\cap E(a,t_0)}\\
 \mbox{($[n]\backslash\{1\}=\cup_{i=1}^{m}A_i$)}&\leq & 1 - \Pr{\sG} + \Pr{\sG}\sum_{i=1}^m\sum_{a\in A_i}\Pr{E(a,t_0)\mid\sG}\\ \mbox{(Prop. \ref{prop:mainbound} + $|A_i|\leq 2^{i}$)}&\leq &1 - \Pr{\sG} + \Pr{\sG}\sum_{i=1}^{+\infty}\left(\frac{2}{5}\right)^i\\ &=& 1 - \frac{\Pr{\sG}}{3}.\end{eqnarray*}
Since $\Pr{\sG}\geq \alpha$ for some universal $\alpha>0$ (cf. \propref{abundance}), we deduce:
$$\Prp{x^{(n)}}{|S^{\sA}_{t_0}|\geq 2}\leq 1-\gamma \mbox{ with }\gamma\equiv\frac{\alpha}{3}\mbox{ universal.}$$
This finishes the proof.

\subsection{Proof of \thmref{Cr}}

We now present the modifications of the previous proof that are
necessary to prove \thmref{Cr}. We keep the definitions from
previous subsections. We will also assume that $k>4$, so that there
exists some $j\in [m]$ with:
$$h\equiv 2^{j+1} -1 = 1+2+\dots+2^j<k/2;$$
in fact, we will assume that $j$ is the {\em largest} number
satisfying this, so that $2^{j+2}\geq k/2$. (The case of $k\leq 4$
follows from \thmref{main}, with an increase in the universal
constant if necessary. If $m$ is too small to allow for this choice of $j$, we may increase $n$ -- and thus $m$ -- at the cost of having more walkers in the beginning of the process.)

We first need an analogue of \propref{simpler}.
\begin{proposition}[Proof omitted]\label{prop:simpler2} Suppose that there exists a universal $\gamma>0$ such that for all $k$ as above, all $n\in\N$ and all $x^{(n)}\in\bV^n$,\begin{equation}\Prp{x^{(n)}}{|S^{\sA}_{t_j}|\geq k+1}\leq 1-\gamma.\end{equation}
Then there exists a universal $K_1>0$ with:
$$\Exp{x^{(n)}}{C_k}\leq K_1\,\left(\frac{\Thit^Q}{k}+\Tmix^Q\right).$$\end{proposition}

We omit the proof of this, which follows that of \propref{simpler} quite closely. The key point is to notice that:
$$t_j = 2\Tmix^Q + \sum_{i=j}^m 2^{4-i}\,(\ln 5)\Thit^Q \leq 2\Tmix^Q + c_1\,2^{-j-1}\,\Thit^Q\leq 2\Tmix^Q + c_2\,\frac{\Thit^Q}{k}$$
with $c_1,c_2>1$ universal (here we used $2^{j+2}\geq k/2$).

We will now bound $\Prp{x^{(n)}}{|S^\sA_{t_j}|\geq k+1}$ in terms of
$h$ and $k$. Using Notational convention \ref{not:convention}, we
first observe that, since $h<k$: $$\Pr{|S^{\sA}_{t_j}|\geq k+1}
=\Pr{\sum_{a=1}^m\Ind{E(a,t_j)}\geq k+1}\leq
\Pr{\sum_{a=h+1}^m\Ind{E(a,t_j)}\geq k-h+1}.$$ Now follow the long
chain of inequalities in the previous subsection to deduce:
\begin{eqnarray*}\Pr{|S^{\sA}_{t_j}|\geq k+1} &\leq & 1-\Pr{\sG} + \Pr{\sG}\,\Pr{\sum_{a=h+1}^m\Ind{E(a,t_j)}\geq k-h+1\mid\sG}\\ \mbox{(Markov ineq.)} &\leq & 1-\Pr{\sG} + \Pr{\sG}\, \frac{\Ex{\sum_{a=h+1}^m\Ind{E(a,t_j)}}}{k-h=1}\\
 \mbox{($[n]\backslash[h]=\cup_{i=j+1}^{m}A_i$)}&\leq & 1 - \Pr{\sG} + \Pr{\sG}\frac{\Ex{\sum_{i=j+1}^m\sum_{a\in A_i}\Ind{E(a,t_j)}}}{k-h} \\ &=& 1 - \Pr{\sG} + \Pr{\sG}\sum_{i=j+1}^m\sum_{a\in A_i}\frac{\Pr{E(a,t_j)}}{k-h+1} \\ \mbox{(Prop. \ref{prop:mainbound} + $|A_i|\leq 2^{i}$)}&\leq &1 - \Pr{\sG} + \Pr{\sG}\sum_{i=j+1}^{+\infty}\frac{2^j\left(\frac{2}{5}\right)^{j-i}}{k-h+1}\\ &=&1-\Pr{\sG} + \Pr{\sG}\frac{2^{j+1}}{3(k-h+1)}\\ \mbox{($2^{j+1}=h+1$)}&\leq & 1- \Pr{\sG} + \Pr{\sG}\frac{h+1}{3(k-h+1)} \\ \mbox{($h<k/2$)}&\leq &= 1- \Pr{\sG} + \Pr{\sG}\frac{k+2}{6(k-k/2)}\\ &=& 1- \Pr{\sG} + \Pr{\sG}\frac{k+2}{3k}\\ \mbox{($k>4$)} &\leq &1- \frac{8}{15}\,\Pr{\sG}.\end{eqnarray*}
To finish, we note that $\Pr{\sG}\geq \alpha>0$ with $\alpha$ universal (\propref{abundance}), hence we may take  $\gamma=8\alpha/15$ in \propref{simpler2}.

\section{On the Meeting Time Lemma}\label{sec:meetingtime}

\subsection{Preliminaries on quasistationary distributions}\label{sec:quasistat}

In this section we review some facts about quasistationary
distributions that will be needed in the proof of \lemref{meetingtime}. We will use the definitions of \secref{markovbasics} throughout the section.

Given any $v\in\bV$, we let ${\qstat}_v$ be a quasistationary
distribution for $\bV\backslash\{v\}$: that is, ${\qstat}_v\in
M_1(\bV)$ satisfies
$$\forall b\in \bV,\, {\qstat}_v(b)=\Prp{{\qstat}_v}{X_t=b\mid H_v>t}.$$

All quasistationary distributions ${\qstat}_v$ corespond to
eigenvalues of restricition of $\Pi^{1/2}Q\Pi^{-1/2}$ to a subspace
$\R^{\bV}_{-v}$ of $\R^{\bV}$ defined below. Here is the recipe.

\begin{enumerate}
\item Consider the subspace:
$$\R^{\bV}_{-v}\equiv \{u\in\R^\bV\,:\,u(v)=0\}$$ and let $\sP_{-v}:\R^{\bV}\to \R^{\bV}_{-v}$ denote the standard projection onto $\R^{\bV}_{-v}$.
$$Q_{-v}\equiv \sP_{-v}\Pi^{1/2} Q\Pi^{-1/2}\sP_{-v}$$ is a
symmetric linear operator from $\R^{\bV}_{-v}$ to itself with
identical diagonal entries and non-positive off-diagonal entries in
the ``obvious" basis for that space, ie. the one given by the
canonical basis vectors $e_b$, $b\in \bV\backslash\{v\}$.
\item By Perron-Frobenius, each irreducible block of the matrix $Q_{-v}$ has a unique eigenvector $w_v\in
\R^{\bV}_{-v}\backslash\{0\}$ with non-negative entries which achieves the smallest eigenvalue $\lambda(w_v)$ corresponding to that block.

\item A simple calculation shows that the vector:
$${\qstat}_v\equiv
\frac{\Pi^{1/2}w_v}{\sum_{b\in\bV}\pi^{1/2}(b)w_v(b)}$$ defines
a probability distribution over $\bV$ with:
$$\Prp{{\qstat}_v}{X_t=b,H_v>t} = {\qstat}_v^\dag \,e^{-t\sP_{-v}Q\sP_{-v}}e_b = e^{-\lambda(w_v) t}{\qstat}_v(b),$$ which in particular implies that
${\qstat}_v$ is a quasistationary distribution associated with $v$.
In particular, $\Exp{{\qstat}_v}{H_v} = 1/\lambda(w_v)>0$. Notice
moreover that ${\qstat}_v(v)=0$.\end{enumerate}

The following
proposition -- an immediate consequence of the third item above -- will be all we need.

\begin{proposition}\label{prop:mineigenvalue}Let $Q_{-v}$ be defined as above and let $\lambda(v)$ denote the smallest eigenvalue of $Q_{-v}$. Then there exists a quasistationary distribution ${\qstat}_v$ for $\bV\backslash\{v\}$ such that $\lambda(v)=1/\Exp{{\qstat}_v}{H_v}$ and:
$$\Prp{{\qstat}_v}{H_v>t} = e^{-\lambda(v) t}.$$\end{proposition}
\begin{proof}This smallest eigenvalue is the smallest eigenvalue of some block of $Q_{-v}$, and thus equals some $w_v$. The rest follows from item $3.$ and from summing the formula for $\Prp{{\qstat}_v}{X_t=b,H_v>t}$ over $b$.\end{proof}

\subsection{Proof of the Meeting Time Lemma}\label{sec:proof_meetingtime}

\begin{proof}[of \lemref{meetingtime}] Fix $n\in\N\backslash\{0\},0<\Delta<n$. We note that:
\begin{multline*}\Prp{\pi}{\forall 0\leq s\leq t,\, X_s\neq h_s} \leq \Prp{\pi}{\cap_{i=1}^{n}\{X(it/n)\neq h(it/n)\}}\\ \leq \Exp{\pi}{\prod_{i=1}^{n}\left(1-\frac{\Delta}{n}\,\Ind{\{X(it/n)=h(it/n)\}}\right)}.\end{multline*}
For a given $v\in\bV$, let $D_v$ be the matrix with a $1$ in
position $(v,v)$ and $0$s elsewhere. A calculation reveals that the
RHS above can be rewritten as:
$$(\Pi{\bf 1})^{\dag}\left\{e^{-\frac{tQ}{n}}\left(I-\frac{\Delta D_{h(t/n)}}{n}\right)\right\}\, \left\{e^{-\frac{tQ}{n}}\left(I-\frac{\Delta D_{h(2t/n)}}{n}\right)\right\}\,\dots\,\left\{e^{-\frac{tQ}{n}}\left(I-\frac{\Delta D_{h(t)}}{n}\right)\right\}\,{\bf 1}$$
where ${\bf 1}$ is the all-ones vector and $\Pi={\rm diag}(\pi(v))_{v\in\bV}$ was introduced in \secref{quasistat}. Since $\Pi$ commutes with all $D_{v}$, we can rewrite the above expression as:
$$(\Pi^{1/2}{\bf 1})^{\dag}\left\{\prod^*_{1\leq i\leq n} e^{-\frac{t\Pi^{1/2}Q\Pi^{-1/2}}{n}}\left(I-\frac{\Delta D_{h(it/n)}}{n}\right)\right\}\Pi^{1/2}{\bf 1}$$
where the $\prod^*$ symbol means that the order of the terms in the product is from left to right.

The vector $\Pi^{1/2}{\bf 1}$ has norm $|\Pi^{1/2}{\bf 1}|^2=\sum_{v}\pi(v)=1$. This implies that the above expression is at most the operator norm of the product of matrices. It follows that:
$$\Prp{\pi}{\forall 0\leq s\leq t,\, X_s\neq h_s} \leq \left\|\left\{\prod^*_{1\leq i\leq n} e^{-\frac{t\Pi^{1/2}Q\Pi^{-1/2}}{n}}\left(I-\frac{\Delta D_{h(it/n)}}{n}\right)\right\}\right\|_{\rm op}$$
Since the operator norm is submultiplicative, we obtain:
\begin{multline}\label{eq:normamax}\Prp{\pi}{\forall 0\leq s\leq t,\, X_s\neq h_s} \leq \prod_{i=1}^n\left\|\left\{e^{-\frac{t\Pi^{1/2}Q\Pi^{-1/2}}{n}}\left(I-\frac{\Delta D_{h(it/n)}}{n}\right)\right\}\right\|_{\rm op}\\ \leq \left(\max_{v\in \bV}\left\|e^{-\frac{t\Pi^{1/2}Q\Pi^{-1/2}}{n}}\left(I-\frac{\Delta D_{v}}{n}\right)\right\|_{\rm op}\right)^n.\end{multline}
We now consider the terms of which we take the maximum in the RHS, for large $n\in\N$. For a given $v\in \bV$, we have:
$$\left\|e^{-\frac{t\Pi^{1/2}Q\Pi^{-1/2}}{n}}\left(I-\frac{\Delta D_{v}}{n}\right) - e^{-\frac{t\Pi^{1/2}Q\Pi^{-1/2}-\Delta D_{v}}{n}}\right\|_{\rm op} =\bigoh{n^{-2}},$$
where the constant implicit in the $\bigoh{n^{-2}}$ term depends
only on $\Delta$, $t$ and $Q$ (and not on $a$, say). Letting $n\to +\infty$ while keeping $\Delta$ fixed, we get:
\begin{multline}\lim_{n\to +\infty}\left(\max_{v\in \bV}\left\|e^{-\frac{t\Pi^{1/2}Q\Pi^{-1/2}}{n}}\left(I-\frac{\Delta D_{v}}{n}\right)\right\|_{\rm op}\right)^n \\ = \lim_{n\to +\infty}\left(\max_{v\in \bV}\left\|e^{-\frac{t\Pi^{1/2}Q\Pi^{-1/2}-\Delta D_v}{n}}\right\|_{\rm op}\right)^n \\ = \max_{v\in \bV}\left\|e^{-t\Pi^{1/2}Q\Pi^{-1/2}-\Delta D_v}\right\|_{\rm op}.\end{multline}
Indeed, last the line follows from the self-adjointness of the exponential and from the fact that $\|B^k\|_{\rm op} = \|B\|^k_{\rm op}$ for self-adjoint matrices $B$. We now use the positive-definiteness of matrix exponentials, together with the spectral mapping property, to deduce:
$$\forall v\in \bV,\,\|e^{-t\Pi^{1/2}Q\Pi^{-1/2}-\Delta D_v}\|_{\rm op} = \lambda_{\max}(e^{-t\Pi^{1/2}Q\Pi^{-1/2}-\Delta D_v}) = e^{-\lambda_{\min}(t\Pi^{1/2}Q\Pi^{-1/2}+\Delta D_v)}.$$
This implies:
$$\Prp{\pi}{\forall 0\leq s\leq t,\, X_s\neq h_s} \leq \exp\left(-\min_{v\in \bV,\Delta>0}\lambda_{\min}(t\Pi^{1/2}Q\Pi^{-1/2}+\Delta D_v)\right).$$
We now make the following Claim.
\begin{claim}As $\Delta\nearrow +\infty$,
$$\lambda_{\min}(t\Pi^{1/2}Q\Pi^{-1/2}+\Delta D_v)\rightarrow t\lambda_{\min}(Q_{-v})$$
where $Q_{-v}$ is defined as in \secref{quasistat}.\end{claim} This
result is probably well-known; for instance, it is a weaker variant
of Lemma 3.1 in \cite{KempeEtAl_2local}. We will prove it below for
completeness, but first we deduce from it that:
$$\Prp{\pi}{\forall 0\leq s\leq t,\, X_s\neq h_s} \leq \exp\left(-t\min_{v\in \bV}\lambda_{\min}(Q_{-v})\right) = e^{-\frac{t}{\Exp{{\qstat}_v}{H_v}}} = \Prp{{\qstat}_v}{H_v>t}$$
via \propref{mineigenvalue}, where ${\qstat}_v$ is some
quasistationary distribution associated with $v$.

We now prove the Claim. Recall the definition of $\sP_{-v}$ in
\secref{quasistat} and notice that $D_v=I-\sP_{-v}$. This shows that
$D_aw=0$ for all $w\in\R^{\bV}_{-v}$ and therefore:
\begin{eqnarray}\nonumber \lambda_{\min}(t\Pi^{1/2}Q\Pi^{-1/2}+\Delta D_v)&=&\inf_{w\in\R^{\bV},\, |w|=1}w^\dag(t
\Pi^{1/2}Q\Pi^{-1/2}+\Delta D_v)\\ \nonumber &\leq&
\inf_{w\in\R^{\bV}_{-v},\, |w|=1}w^\dag(t
\Pi^{1/2}Q\Pi^{-1/2}+\Delta D_v)w \\ \nonumber &=&
\inf_{w\in\R^{\bV}_{-v},\, |w|=1}w^\dag(t \Pi^{1/2}Q\Pi^{-1/2})w\\
\nonumber \mbox{ (use $\sP_{-v}w=w$) } &=&
t\inf_{w\in\R^{\bV}_{-v},\, |w|=1}w^\dag(\sP_{-v}
\Pi^{1/2}Q\Pi^{-1/2}\sP_{-v})w\\ \label{eq:uppermeetingtime}&=&
t\lambda_{\min}(Q_{-a}).\end{eqnarray} To get an opposite
inequality, we set $A=t\Pi^{1/2}Q\Pi^{-1/2}$ for convenience. We
first show that there exists some $c>0$ such that for all large
enough $\Delta>0$,
\begin{equation}\label{eq:order}A+\Delta D_v \succeq
\sP_{-v}A\sP_{-v} - \frac{c\sP_{-v}}{\Delta} +
\frac{\Delta}{2}\,D_v,\end{equation} where for symmetric matrices
$B_1,B_2$ with the same size, $B_1\preceq B_2$ means that $B_2-B_1$
is positive semidefinite. To see this, we use $\sP_{-v}+D_v=I$
several times and notice that for any $x\in\R^{\bV}$,
\begin{eqnarray*}x^{\dag}(A+\Delta D_v)x  &=& x^\dag \sP_{-v}A\sP_{-v}x  + x^\dag D_vAD_vx+ 2x^\dag (\sP_{-v}AD_v)x \\ & & +\Delta (x^\dag
D_vx) \\ \mbox{(Cauchy-Schwartz)}&\geq & x^\dag \sP_{-v}A\sP_{-v}x +
x^\dag \frac{\Delta D_v}{2} x +
|D_vx|^2\left(\frac{\Delta}{2}-\|A\|_{\rm op}\right) \\ & & - 2\|A\|_{\rm op}\,|\sP_{-v}x|||D_vx|\\
\mbox{(assume $\Delta>4\|A\|_{\rm op}$)}&\geq &
x^\dag\left(\sP_{-v}A\sP_{-v} + \frac{\Delta D_v}{2}\right)x \\ & &
+
\left(\frac{\sqrt{\Delta}|D_vx|}{2} - \frac{2\|A\|_{\rm op}|\sP_{-v}x|}{\sqrt{\Delta}}\right)^2 - \frac{4\|A\|_{\rm op}|\sP_{-v}x|^2}{\Delta}\\
\mbox{(set $c\equiv 4\|A\|_{\rm op}$)}&\geq &
x^\dag\left(\sP_{-v}A\sP_{-v} - \frac{c\sP_{-v}}{\Delta} +
\frac{\Delta D_v}{2}\right)x.\end{eqnarray*} This proves
\eqnref{order}, which implies:
\begin{equation}\label{eq:order2}\lambda_{\min}(A+\Delta D_v)\geq
\lambda_{\min}\left(\sP_{-v}A\sP_{-v} - \frac{c\sP_{-v}}{\Delta} +
\frac{\Delta}{2}\,D_v,\right).\end{equation} Notice that the matrix
in the RHS has $\R^{\bV}_{-v}$ as an invarant subspace, which
implies that all of its eigenvectors lie in $\R^{\bV}_{-v}$ or in
its orthogonal complement. It is easy to see that the all vectors in
the latter space are eigenvectors with eigenvalue $\Delta/2$;
therefore, for all large enough $\Delta$ the {\em minimal}
eigenvalue corresponds to a vector in $\R^{\bV}_{-v}$. We deduce
that for all large $\Delta>0$,
\begin{multline*}\lambda_{\min}\left(\sP_{-v}A\sP_{-v} -
\frac{c\sP_{-v}}{\Delta} + \frac{\Delta}{4}\,D_v,\right)
\\=\min_{w\in\R^{\bV}_{-v},|w|=1}w^{\dag}\left(\sP_{-v}A\sP_{-v} -
\frac{c\sP_{-v}}{\Delta}\right)w = t\lambda_{\min}(Q_{-v}) -
\frac{2c}{\Delta}\end{multline*} because $w^{\dag}\sP_{-v}A\sP_{-v}w
= tw^\dag Q_{-v}w$ for all $w$ as above. Together with
\eqnref{uppermeetingtime} and \eqnref{order2}, this shows that:
$$\mbox{For large enough $\Delta>0$, }t\lambda_{\min}(Q_{-v}) - \frac{c}{\Delta}\leq \lambda_{\min}(t\Pi^{1/2}Q\Pi^{-1/2}+\Delta
D_v)\leq t\lambda_{\min}(Q_{-v}),$$ and the Claim follows when we
let $\Delta\nearrow +\infty$.\end{proof}

\section{Final remarks}

\begin{itemize}
\item Let ${\rm T}_{\rm meet}^Q$ denote the maximum expected meeting time of two independent realizations of $Q$. In light of the discussion in the Introduction, it would be natural to expect that $\Ex{C}\leq K_2\,{\rm T}_{\rm meet}^Q$ for some universal $K_2>0$ and all $Q$. Is this actually true? A more modest question is whether the constants in the two Theorems can be improved.
\item The Meeting Time Lemma (\lemref{meetingtime}) can be used in the study of a cat-and-mouse game proposed in \cite[Chapter 4, page 17]{AldousFill_RWBook}. In this game a cat moves according to a reversible Markov chain $Q$. A mouse chooses a trajectory $(h_s)_{s\geq 0}$ for itself and an initial distribution for the cat. Aldous and Fill asked if staying put at some carefully chosen state gives an optimal strategy for the mouse in terms of maximizing $\Ex{M}$, where $M$ is the meeting time of cat and mouse. One can use \lemref{meetingtime} to prove that if $\Tmix^Q\ll \max_{v}\Exp{\pi}{H_v}$ (a natural condition in many examples), then the strategy where the mouse stays at $v$ and chooses $\qstat_v$ as the initial distribution nearly maximizes $\Pr{M>t}$ simultaneously for all $t\geq 0$. We expect to comment on this and related results in a upcoming note.
\end{itemize}

\bibliography{bibfile_rimfo}
\bibliographystyle{plain}
\end{document}